\documentclass[10pt]{amsart}
\usepackage{graphicx}
\usepackage{amsmath, amsthm,  amsfonts, amscd, parskip} 
\usepackage{url}
\usepackage{tikz-cd}
\usepackage[letterpaper, margin=3cm]{geometry}
\usepackage{slashed} 
\usepackage{amssymb} 
\allowdisplaybreaks 

\DeclareMathOperator{\id}{id}

\theoremstyle{plain}
\newtheorem{theorem}{Theorem}[section]
\newtheorem{lemma}[theorem]{Lemma}
\newtheorem{corollary}[theorem]{Corollary}
\newtheorem{proposition}[theorem]{Proposition}
\newtheorem{question}[theorem]{Question}

\theoremstyle{definition}
\newtheorem{definition}[theorem]{Definition}
\newtheorem{example}[theorem]{Example}
\newtheorem{remark}[theorem]{Remark}

\numberwithin{equation}{section}
\numberwithin{figure}{section}

\newcommand{\CC}{{\mathbb C}}
\newcommand{\RR}{{\mathbb R}}
\newcommand{\ZZ}{{\mathbb Z}}
\newcommand{\QQ}{{\mathbb Q}}

\newcommand{\CP}{{\mathbb {CP}}}
\newcommand{\HP}{{\mathbb{HP}}}

\DeclareMathOperator{\ch}{{\mathrm{ch}}}

\title{Spin$^h$ and further generalisations of spin} 
\author{Michael Albanese, Aleksandar Milivojevi\'c}

\vspace{1em}
\address{University of Waterloo}
\email{m3albane@uwaterloo.ca}
\address{Max Planck Institute for Mathematics} 
\email{milivojevic@mpim-bonn.mpg.de}

\begin{document}

\begin{abstract}
The question of which manifolds are spin or spin$^c$ has a simple and complete answer. In this paper we address the same question for spin$^h$ manifolds, which are less studied but have appeared in geometry and physics in recent decades. We determine that the first obstruction to being spin$^h$ is the fifth integral Stiefel--Whitney class $W_5$. Moreover, we show that every compact orientable manifold of dimension 7 or lower is spin$^h$, and that there are orientable manifolds which are not spin$^h$ in all higher dimensions. We are then led to consider an infinite sequence of generalised spin structures. In doing so, we show that there is no integer $k$ such that every manifold embeds in a spin manifold with codimension $k$.
\end{abstract}\

\maketitle

\section{Introduction}

Every smooth manifold embeds in an orientable manifold with codimension one. Namely, a manifold $M$ embeds in the total space of its determinant line bundle $L$, which is orientable as $w_1(L) =w_1(M)$.

Analogously, we could consider the following:

\begin{question}\label{Q1}
Is there a positive integer $k$ such that every manifold embeds in a spin manifold with codimension $k$?
\end{question}

By the Whitney embedding theorem, every $n$--dimensional manifold embeds in $\mathbb{R}^{2n}$ which is spin, but this does not provide an answer to the question as the codimension depends on $n$. As above, we look for a vector bundle $E \to M$ whose total space is spin. Upon passing to a manifold of one larger dimension if necessary, we may assume $M$ is orientable, and so we aim to find an orientable vector bundle $E$ with $w_2(E) = w_2(M)$. So we could instead answer the following:

\begin{question}\label{Q2} 
Is there a positive integer $k$ such that every orientable manifold $M$ admits an orientable (real) rank $k$ vector bundle $E$ with $w_2(E) = w_2(M)$?
\end{question}

In fact, this question is equivalent to Question \ref{Q1}: given an embedding in a spin manifold with codimension $k$, then one can take $E$ to be the normal bundle. Unlike in the orientable analogue, there is no canonical choice of bundle $E$ with $w_2(E) = w_2(M)$. 

Note that if $k = 2$ were an answer to Question \ref{Q2}, then we would deduce that every orientable manifold is spin$^c$. This is not the case as the 5--dimensional Wu manifold $SU(3)/SO(3)$ illustrates. However, this manifold does admit what is known as a spin$^h$ structure. An orientable $n$-manifold admits a spin$^h$ structure if the structure group of its tangent bundle lifts from $SO(n)$ to $Spin^h(n) = Spin(n) \times_{\ZZ_2} Sp(1)$; see section 2. If every orientable manifold were to admit such a structure, then $k = 3$ would be an answer to Question \ref{Q2}.

In section 2, we study the existence of spin$^h$ structures and in doing so, prove the following (Corollary \ref{spinhW5}):

\begin{theorem}
The primary obstruction to admitting a spin$^h$ structure is the fifth integral Stiefel--Whitney class $W_5$.
\end{theorem}

Moreover, we obtain an explicit integral lift of $w_4$ in terms of the spin$^h$ structure. In section 4 (Remark \ref{obstruction}), we observe that $W_5$ is not the only obstruction to a spin$^h$ structure.

The above obstruction vanishes on any orientable 5--manifold, hence any such manifold admits a spin$^h$ structure. We improve on this by showing (Theorem \ref{upto7}):

\begin{theorem}
The following hold: \begin{enumerate} \item Every (not necessarily compact) orientable manifold of dimension $\leq 5$ is spin$^h$. 
\item Compact orientable manifolds of dimension $6$ and $7$ are spin$^h$. \item A non-compact orientable manifold $M$ of dimension $6$ or $7$ is spin$^h$ if and only if $W_5(TM) = 0$. \item A non-compact orientable manifold $M$ of dimension $6$ or $7$ with no elements of order exactly four in $H^5(M;\ZZ)$ is spin$^h$. \end{enumerate}
\end{theorem}

In particular, closed orientable manifolds of dimension $\leq 7$ are spin${}^h$. In section 4 (Theorem \ref{nonspinhall}), we see that this is sharp:

\begin{theorem}
In every dimension $\geq 8$, there are infinitely many homotopy types of closed smooth simply connected manifolds which do not admit spin$^h$ structures.
\end{theorem}

In section 3, we consider the groups $Spin^k(n) := (Spin(n)\times Spin(k))/\mathbb{Z}_2$ and their associated structures which generalise spin, spin$^c$, and spin$^h$ structures. The existence of such a structure on a smooth manifold $M$ is equivalent to the existence of an orientable rank $k$ vector bundle $E$ with $w_2(E) = w_2(M)$, see Proposition \ref{equivalences}. By using techniques from rational surgery theory, we prove the following (Theorem \ref{4}), which provides a negative answer to Questions \ref{Q1} and \ref{Q2}:

\begin{theorem}
For every positive integer $k$, there exists a closed smooth simply connected manifold which does not admit a spin$^k$ structure.
\end{theorem}

In section 5, we conclude with a brief discussion of non-orientable analogues of spin$^k$ structures.

\subsection*{Acknowledgements} The authors were introduced to spin$^h$ structures by Xuan Chen \cite{Chen}, who raised the question of whether every orientable 5--manifold is spin$^h$, and whether there exist orientable manifolds which are not spin$^h$. The same questions arose in a conversation between Edward Witten and the first author, as well as questions about non-orientable analogues of the spin$^h$ condition.

The authors would like to thank Blaine Lawson for some interesting discussions during the early development of these ideas, Claude LeBrun for his help with Example \ref{AQ}, and  Gustavo Granja and Peter Teichner for helpful conversations. We would like to thank the Instituto Superior T\'ecnico Lisboa, the Max-Planck-Institut f\"ur Mathematik Bonn, and the University of Waterloo for their generous hospitality during our respective visits.

\section{Spin${}^h$ structures: existence and non-existence results}

We assume the reader is familiar with spin and spin$^c$ structures; see e.g. \cite{LM} or \cite{Friedrich} for a detailed discussion. Throughout, we will not distinguish between principal $SO(k)$-bundles and their naturally associated orientable rank $k$ vector bundles.

In a similar spirit to the groups $Spin^c(n) = (Spin(n) \times U(1))/\ZZ_2$, the groups $Spin^h(n)$ are defined to be $(Spin(n)\times Sp(1))/\mathbb{Z}_2$ where, in both cases, the $\mathbb{Z}_2$ subgroup is generated by $(-1, -1)$; the exponent $h$ is used to highlight the role of the quaternions $\mathbb{H}$. There is a natural group homomorphism $Spin^h(n) \to SO(n)\times SO(3)$ which is a double covering, and by composing with projections, there are induced homomorphisms $Spin^h(n) \to SO(n)$ and $Spin^h(n) \to SO(3)$.

\begin{definition}
A \textit{spin$^h$ structure} on a principal $SO(n)$-bundle $P$ is a principal $SO(3)$-bundle $E$ together with a principal $Spin^h(n)$-bundle $Q$ and a double covering $Q \to P\times E$ which is equivariant with respect to $Spin^h(n) \to SO(n)\times SO(3)$. We call $E$ the \textit{canonical $SO(3)$-bundle} of the spin$^h$ structure. \end{definition}

It follows from the definition that for the canonical $SO(3)$-bundle $E$ we have $w_2(E) = w_2(P)$. Conversely, if there is a principal $SO(3)$-bundle $E$ with $w_2(E) = w_2(P)$, then $P$ admits a spin${}^h$ structure with canonical $SO(3)$-bundle $E$, see e.g. \cite[Proposition 2.3]{Nagase}.

One can also define a spin$^h$ structure on a principal $SO(n)$-bundle $P$ as a reduction of structure group with respect to the homomorphism $p \colon Spin^h(n) \to SO(n)$. The canonical $SO(3)$-bundle of the spin$^h$ structure is then the bundle associated to the homomorphism $Spin^h(n) \to SO(3)$. Note that there are inclusions $Spin(n) \to Spin^c(n) \to Spin^h(n)$ which commute with the projections to $SO(n)$, so a spin or spin$^c$ structure induces a spin$^h$ structure. The canonical $SO(3)$-bundle of the corresponding spin$^h$ structure is trivial in the spin case, and of the form $L\oplus\varepsilon^1$ in the spin$^c$ case, where $L$ is the canonical $SO(2)$-bundle of the spin$^c$ structure and $\varepsilon^1$ denotes the trivial bundle of rank 1. 

A spin$^h$ structure on an oriented Riemannian manifold $(M, g)$ is defined to be a spin$^h$ structure on the bundle of oriented orthonormal frames. As in the spin and spin$^c$ cases, there are natural identifications between spin$^h$ structures for different orientations and metrics, so we do not require these choices when speaking of a spin$^h$ structure on an orientable manifold $M$. We call $M$ spin$^h$ if it admits such a structure.

The notion of spin$^h$ structures appears in \cite{Bar}, but they also appear in \cite{Nagase} under the name spin$^q$. They later featured in the thesis \cite{Chen} and in the physics literature, where instead of $Sp(1)$, the isomorphic Lie group $SU(2)$ is preferred. The relevant groups were named $G_0(n)$ in \cite[Proposition 9.16]{FH}, see also \cite[Appendix D]{SSGR}. In \cite{LS19}, this notion shows up in the more general context of Lipschitz structures. When $n = 4$, what we call $Spin^h(4)$ was denoted by $\overline{Spin}(4)$ in \cite{BFF} and $Spin_{SU(2)}(4)$ in \cite{WWW}; the group was also alluded to in \cite{HP}. The concept of a spin$^h$ structure on a 4--manifold was used to construct a quaternionic version of Seiberg--Witten theory in \cite{OT96}.

\begin{example}\label{AQ}
Recall, a $4n$--dimensional manifold $M$ is called \textit{almost quaternionic} if its structure group can be reduced to $GL(n, \mathbb{H})\cdot\mathbb{H}^* \cong (GL(n, \mathbb{H})\times\mathbb{H}^*)/\mathbb{Z}_2$ where $\mathbb{Z}_2$ acts diagonally by $(-1, -1)$. Choosing a compatible metric, one can always reduce further to the maximal compact subgroup $Sp(n)\cdot Sp(1)$. An equivalent definition is a manifold $M$ together with a rank three subbundle $Q$ of $\operatorname{End}(TM)$ which has a local basis of sections $I, J, K$ satisfying the quaternionic relations $I^2 = J^2 = K^2 = IJK = -\operatorname{id}$. The bundle $Q$ is the associated bundle of the homomorphism $Sp(n)\cdot Sp(1) \to SO(3)$ given by projecting the second factor. The natural spin$^c$ structure on an almost complex manifold has canonical $SO(2)$-bundle the anticanonical bundle\footnote{There are actually two natural spin$^c$ structures associated to an almost complex structure (due to the fact that inversion is an automorphism of $U(1)$), one for which the canonical $SO(2)$-bundle is the anticanonical bundle, and one for which it is the canonical bundle. While the latter would provide more consistency of the term `canonical', the former is more prevalent in the literature.}; does an almost quaternionic manifold admit a spin$^h$ structure for which the canonical $SO(3)$-bundle is $Q$? 

The map $Sp(n)\cdot Sp(1) \to SO(4n)$ lifts to $Spin(4n)$ if $n$ is even; let $n = 2m$. Therefore, every $8m$--dimensional almost quaternionic manifold has a canonical spin structure. The induced spin$^h$ structure has trivial canonical $SO(3)$-bundle, however the bundle $Q$ need not be trivial for an $8m$--dimensional almost quaternionic manifold, as $\mathbb{HP}^{2m}$ demonstrates. One might wonder then whether there is an alternative construction of a spin$^h$ structure which would have the desired canonical $SO(3)$-bundle. This is not possible as there exist $8m$--dimensional almost quaternionic manifolds with $w_2(Q) \neq 0 = w_2(M)$, such as $M = \operatorname{Gr}_2(\mathbb{C}^{2m+2})$; the class $w_2(Q)$ is known as the Marchiafava--Romani class. More generally, a theorem of Salamon \cite[Theorem 6.3]{Sal} states that a quaternion-K\"ahler manifold with positive scalar curvature and zero Marchiafava--Romani class is isometric to quaternionic projective space. In particular, every quaternion-K\"ahler symmetric space other than quaternionic projective space has $w_2(Q) \neq 0$.

On the other hand, if $n$ is odd, say $n = 2m + 1$, we have $w_2(Q) = w_2(M)$ \cite{MR}, so one might hope that such a spin$^h$ structure exists. This is indeed true, see \cite[Theorem 3.3]{Nagase} or \cite[Lemma 2]{Bar}. Note that when $m = 0$, we are in the 4--dimensional case and $Sp(1)\cdot Sp(1) = SO(4)$, so every orientable 4--manifold is almost quaternionic. Moreover, there are two natural spin$^h$ structures (corresponding to the two factors of $Sp(1)$), and the canonical $SO(3)$-bundles are $\Lambda^+$ and $\Lambda^-$, which denote the bundles of self-dual and anti-self-dual two-forms respectively.\end{example}

Every orientable manifold of dimension $\leq 3$ is spin (with $\CP^2$ being an orientable non-spin 4--manifold), and every orientable manifold of dimension $\leq 4$ is spin${}^c$ \cite{TV} (with the Wu manifold $SU(3)/SO(3)$ being an orientable non-spin${}^c$ 5--manifold \cite[p. 50]{Friedrich}). The Wu manifold does however admit a spin$^h$ structure, with canonical $SO(3)$-bundle $SO(3) \to SU(3) \to SU(3)/SO(3)$, see e.g. \cite[Example p.26]{Chen}. As far as the authors are aware, it is unknown in which range of dimensions the orientability of a manifold implies the existence of a spin${}^h$ structure. This naturally leads us to consider the obstructions to admitting a spin${}^h$ structure.

Recall that a principal $SO(n)$-bundle $P$ admits a spin structure if and only if $w_2(P) = 0$, while it admits a spin$^c$ structure if and only if there is a principal $SO(2)$-bundle $E$ with $w_2(E) = w_2(P)$. As $SO(2) = U(1)$, the bundle $E$ has a first Chern class $c_1(E)$ which satisfies $c_1(E) \equiv w_2(E) \bmod 2$, so we see that a necessary condition for $P$ to admit a spin$^c$ structure is that $w_2(P)$ admits an integral lift. This is also sufficient as every degree 2 integral cohomology class is the first Chern class of a complex line bundle. An equivalent reformulation of this condition is that $W_3(P) = 0$ where $W_3(P) = \beta w_2(P)$ is the third integral Stiefel--Whitney class of $P$; here $\beta$ denotes the integral Bockstein map $\beta \colon H^*(-;\ZZ_2) \to H^{*+1}(-;\ZZ)$.

As there are cohomology classes whose vanishing detects whether a principal $SO(n)$-bundle $P$ admits a spin or spin$^c$ structure, namely $w_2(P)$ and $W_3(P)$ respectively, one is naturally led to ask which cohomology classes associated to $P$ must vanish if it admits a spin$^h$ structure.

\begin{proposition}\label{integralliftw4} Let $P$ be a principal $SO(n)$-bundle over a CW complex $X$. If $P$ admits a spin$^h$ structure, then $W_5(P) = 0$. If $E$ is the canonical $SO(3)$-bundle associated to a spin${}^h$ structure, then an integral lift of $w_4(P)$ is given by $\tfrac{1}{2}\left( p_1(P) - p_1(E) \right)$. \end{proposition}

\begin{proof}

We assume $n \geq 5$ since otherwise $W_5 = 0$ for any orientable vector bundle (for $n\leq 3$ this follows since $W_5 = \beta w_4 = 0$, and for $n=4$ the Euler class is an integral lift of $w_4$, and hence $W_5 = \beta w_4 = 0$). We will show that $W_5$ pulls back to 0 under the map $p \colon BSpin^h(n) \to BSO(n)$ induced by the projection $Spin^h(n) \to SO(n)$. Note that by assumption the map $X \to BSO(n)$ classifying $P$ lifts through the map $p$, and so the claim will follow.

Let $\hat{P}$ denote the $SO(n)$-bundle classified by the map $p \colon BSpin^h(n) \to BSO(n)$. By construction, $\hat{P}$ has a spin$^h$ structure with canonical $SO(3)$-bundle classified by the map $BSpin^h(n) \to BSO(3)$ induced by the projection $Spin^h(n) \to SO(3)$. Now consider the Pontryagin square operation $\mathfrak{P} \colon H^2(-; \ZZ_2) \to H^4(-;\ZZ_4)$, which by a theorem of Wu satisfies $$\mathfrak{P}(w_2(V)) = \rho_4(p_1(V)) + i_*(w_4(V))$$ for any orientable vector bundle $V$, where $\rho_4$ and $i_*$ denote the maps on cohomology corresponding to the coefficient homomorphisms $\ZZ \to \ZZ_4$ given by reduction mod 4 and the inclusion $i \colon \ZZ_2 \to \ZZ_4$ respectively, see \cite[Theorem C]{Th60}. Since $w_2(\hat{P}) = w_2(\hat{E})$ and $w_4(\hat{E}) = 0$, applying the above formula we have \begin{align*}  i_*(w_4(\hat{P})) &= \mathfrak{P}(w_2(\hat{P})) - \rho_4(p_1(\hat{P})) \\ &= \mathfrak{P}(w_2(\hat{E})) - \rho_4(p_1(\hat{P})) \\ &= \rho_4(p_1(\hat{E})) - \rho_4(p_1(\hat{P})),\end{align*} i.e. $i_*(w_4(\hat{P}))$ has an integral lift, namely $p_1(\hat{E}) - p_1(\hat{P})$.

Now consider the map of short exact sequences $$\begin{tikzcd}
0 \arrow[r] \arrow[d] & \ZZ \arrow[d] \arrow[r] & \ZZ \arrow[d] \arrow[r] & \ZZ_2 \arrow[d] \arrow[r] & 0 \arrow[d] \\
0 \arrow[r]           & \ZZ_2 \arrow[r]         & \ZZ_4 \arrow[r]         & \ZZ_2 \arrow[r]           & 0          
\end{tikzcd}$$

and the induced map of long exact sequences in cohomology $$\begin{tikzcd}
H^3(BSpin^h(n);\ZZ_2) \arrow[r] \arrow[d, "="] & H^4(BSpin^h(n);\ZZ) \arrow[d, "\rho_2"] \arrow[r] & H^4(BSpin^h(n);\ZZ) \arrow[d, "\rho_4"] \arrow[r] & H^4(BSpin^h(n);\ZZ_2) \arrow[d, "="]  \\
H^3(BSpin^h(n);\ZZ_2) \arrow[r, "Sq^1"]   & H^4(BSpin^h(n);\ZZ_2) \arrow[r, "i_*"]       & H^4(BSpin^h(n);\ZZ_4) \arrow[r]              & H^4(BSpin^h(n);\ZZ_2)               
\end{tikzcd}$$

where $\rho_2$ is induced by reduction mod 2. Since $i_*(w_4(\hat{P}))$ maps to zero in $H^4(BSpin^h(n);\ZZ_2)$ by exactness, we see that $p_1(\hat{E}) - p_1(\hat{P})$ also maps to zero by commutativity. Therefore, there is a class in $H^4(BSpin^h(n);\ZZ)$ mapping to $p_1(\hat{E}) - p_1(\hat{P})$ under multiplication by two.

After converting $p \colon BSpin^h(n) \to BSO(n)$ to a fibration, it follows from the Serre spectral sequence that $H^4(BSpin^h(n);\ZZ) \cong \ZZ \oplus \ZZ$. Indeed, $H^4(BSO(n);\ZZ) \cong \ZZ$ and the homotopy fiber of $p$ is the classifying space of the kernel of the map $Spin^h(n) \to SO(n)$, i.e. $BSp(1) \simeq \HP^{\infty}$. Since $H^5(BSO(n);\ZZ) \cong \ZZ_2$, generated by $W_5$ \cite[Theorem 1.5]{Brown}, either the degree 4 generator of $H^*(\HP^\infty; \ZZ)$ (which is unique up to sign) is closed throughout the spectral sequence, or it trangresses to $W_5$, in which case twice the generator will be closed. (We will argue later that it is in fact the latter that happens.) In either case, $E_\infty^{0,4} \cong \ZZ$. The extension problem in the spectral sequence thus has a unique solution of $H^4(BSpin^h(n);\ZZ) \cong \ZZ \oplus \ZZ$. In particular, as there is no 2--torsion, we may unambigiously write $\tfrac{1}{2}\left( {p_1(\hat{E}) - p_1(\hat{P})}\right)$ for the class mapping to $p_1(\hat{E}) - p_1(\hat{P})$.

By commutativity of the diagram, we now have $\rho_2\left( \tfrac{1}{2}\left( p_1(\hat{E}) - p_1(\hat{P})\right) \right) - w_4(\hat{P}) \in \ker i_*$, and so there is an element $b \in H^3(BSpin^h(n);\ZZ_2)$ such that $w_4(\hat{P}) = \rho_2\left( \tfrac{1}{2}\left(p_1(\hat{E}) - p_1(\hat{P})\right) \right) + Sq^1b$. It is true that $Sq^1 b$ has an integral lift in general (namely the Bockstein of $b$), but we will see in fact that this term vanishes. To this end, notice that $p^* \colon H^3(BSO(n);\ZZ_2) \to H^3(BSpin^h(n);\ZZ_2)$ is an isomorphism since the homotopy fiber $\HP^\infty$ of $p$ is 3--connected. Since $H^3(BSO(n);\ZZ_2)$ is spanned by $w_3$, and $Sq^1 w_3 = 0$ (the Bockstein of $w_2$, i.e. $W_3$, is an integral lift of $w_3$), by the naturality of $Sq^1$ we conclude that $Sq^1$ vanishes on $H^3(BSpin^h(n);\ZZ_2)$. Therefore $\tfrac{1}{2}\left(p_1(\hat{E}) - p_1(\hat{P})\right)$ is an integral lift of $w_4$, and so is $\tfrac{1}{2}\left( p_1(\hat{P}) - p_1(\hat{E})\right)$. The classes $p_1(\hat{P})$ and $p_1(\hat{E})$ pull back to $p_1(P)$ and $p_1(E)$ on $X$, respectively, and so $w_4(P)$ has integral lift $\tfrac{1}{2}\left( p_1(P) - p_1(E)\right)$; in particular $W_5(P) = 0$. \end{proof}

\begin{corollary}\label{W5} On a spin${}^h$ manifold, $W_5 = 0$.\end{corollary}

Note that this gives a topological restriction on almost quaternionic manifolds. In the $8m$--dimensional case, almost quaternionic manifolds are spin and such manifolds are known to have $W_5 = 0$, see Remark \ref{integrallifts}.

\begin{corollary}\label{spinhW5} The primary obstruction to a principal $SO(n)$-bundle $P$ admitting a spin${}^h$ structure is $W_5(P)$. \end{corollary} 

\begin{proof} Since the homotopy fiber $\HP^\infty$ of $p \colon BSpin^h(n) \to BSO(n)$ is 3--connected, the primary obstruction to lifting is the pullback of a class in $H^5(BSO(n);\pi_4(\HP^\infty)) \cong H^5(BSO(n);\ZZ)$, as can be seen from the Moore--Postnikov tower for $p$ \cite[Section 8.3]{Span89}. Recall, $H^5(BSO(n);\ZZ)$ is spanned by the order two element $W_5$. If the primary obstruction to a spin${}^h$ structure were 0, then from the Moore--Postnikov tower for $p$ we see that $p^*W_5$ would be non-zero in $H^5(BSpin^h(n);\ZZ)$; however $p$ classifies the principal $SO(n)$-bundle associated to the tautological $Spin^h(n)$ bundle over $BSpin^h(n)$, and so as we saw above, $p^*W_5$ must vanish as this bundle admits a $Spin^h$ structure. In particular this shows that the degree $4$ generator of $H^*(\HP^\infty;\ZZ)$ trangresses to $W_5$ in the Serre spectral sequence, as commented earlier. \end{proof}

\begin{proposition}\label{5manifolds} Every orientable manifold of dimension 5 is spin$^h$. \end{proposition}

\begin{proof}
Any orientable 5--manifold $M$ has torsion-free $H^5(M;\ZZ)$ (namely $\ZZ$ in the closed case and $0$ otherwise). Since the primary obstruction to a spin${}^h$ structure $W_5$ is torsion, it vanishes on $M$. The higher obstructions to a spin${}^h$ structure lie in $H^{\ast \geq 6}(M;\pi_{\ast-1}\HP^\infty)$ and so they all vanish on $M$ as well.
\end{proof}

\begin{remark}\label{integrallifts} In the statement of Proposition \ref{integralliftw4}, we chose to denote the integral lift of $w_4$ as $\tfrac{1}{2}\left(p_1(P) - p_1(E)\right)$ instead of its negative. In this form the integral lift visibly generalises the known integral lift of $w_4$ of the tangent bundle on spin and spin${}^c$ manifolds. For spin manifolds, it is known that an integral lift of $w_4$ is provided by $\tfrac{1}{2}p_1$. For spin${}^c$ manifolds, consider the fibration $\CP^{\infty} \to BSpin^c(n) \to BSO(n)$, and denote the degree 2 generator of $H^*(\CP^\infty; \ZZ)$ by $x$. The transgression of $x$ in the Serre spectral sequence is $W_3$. Denoting by $c$ the unique integral cohomology class in $BSpin^c$ that restricts to $2x$ on the fiber, we have $\tfrac{1}{2}\left( p_1 - c^2\right) \equiv w_4 \bmod 2$ in $BSpin^c$ \cite[Eqs. 5.2 and 5.3]{HHLZ20}; note that $c$ classifies the canonical $SO(2)$-bundle determined by the projection $Spin^c(n) \to SO(2)$. 

If $L$ is the complex line bundle associated to a spin$^c$ structure, then since $p_1(L) = c_1(L)^2$, we see that the integral lift of $w_4$ given by the spin$^c$ structure is consistent with the integral lift given by the induced spin$^h$ structure; likewise for spin.
\end{remark}

In the case of spin and spin${}^c$ structures on an orientable manifold, there is a single obstruction to admitting such a structure, namely $w_2$ and $W_3$ respectively. However, there is no reason to expect that the primary obstruction $W_5$ is the only obstruction to a spin${}^h$ structure, as the homotopy fiber of the relevant obstruction theoretic problem is not an Eilenberg--MacLane space; note that the homotopy fiber of $BSpin(n) \to BSO(n)$ is $B\mathbb{Z}_2 = K(\mathbb{Z}_2, 1)$ and the homotopy fiber of $BSpin^c(n) \to SO(n)$ is $BU(1) = K(\mathbb{Z}, 2)$. In fact, there are further obstructions to spin$^h$ structures; see Remark \ref{obstruction}. Due to the lack of a simple criterion for the existence of spin$^h$ structures, to the best of the authors' knowledge, there are no known examples of orientable manifolds which are not spin$^h$. Below, we consider an explicit manifold with $W_5 \neq 0$, and hence by Corollary \ref{W5}, it provides an example of an orientable non-spin$^h$ manifold.

\begin{example}\label{WuxWu} Consider the product $W\times W$ of the 5--dimensional Wu manifold $W = SU(3)/SO(3)$ with itself. If $W\times W$ were to admit a spin${}^h$ structure, then $w_4(W\times W)$ would admit an integral lift. Denoting the projections onto the two factors of $W$ by $\pi_1$ and $\pi_2$, we have $w_4(W\times W) = \pi_1^*w_2(W)\pi_2^*w_2(W)$. Note that $H^4(W\times W;\ZZ) = 0$ by the K\"unneth formula, so if $\pi_1^*w_2(W)\pi_2^*w_2(W)$ were to admit an integral lift it would be zero. However, $\pi_1^*w_2(W)\pi_2^*w_2(W)$ is the unique non-zero element in $H^4(W\times W;\ZZ_2)$ since $w_2(W)$ is the unique non-zero element in $H^2(W;\ZZ_2)$. Therefore $W_5(W\times W) \neq 0$. 

Notice that this example also shows that the product of two spin${}^h$ manifolds need not be spin${}^h$, in contrast to the case of spin and spin${}^c$ manifolds.
\end{example}

Naturally one is led to ask whether there are non-spin${}^h$ manifolds of smaller dimension. We will address this in the following sections.

\section{Spin${}^k$ structures}

The existence of orientable manifolds which are not spin$^h$ can be viewed as motivation to consider further generalisations of spin structures. Note that $Spin(n) \cong (Spin(n)\times O(1))/\mathbb{Z}_2$ while $Spin^c(n) = (Spin(n)\times U(1))/\mathbb{Z}_2$ and $Spin^h(n) = (Spin(n)\times Sp(1))/\mathbb{Z}_2$, so one can regard these three families of groups as real, complex, and quaternionic families respectively. It is natural to ask whether an octonionic analogue exists, at least for some values of $n$. The natural inclination is to replace the second factor with the collection of automorphisms of the octonions $\mathbb{O}$ which preserve the hermitian form $\langle x, y\rangle = x\overline{y}$. This collection is precisely $S(\mathbb{O})$, the octonions of unit norm. Due to the failure of associativity, this is not a group but rather a Moufang loop. One could attempt to circumvent this difficulty in a variety of ways, but we will not do so here.

There is another way that one can generalise the three families of groups above. Note that 
\begin{align*}
Spin(n) &\cong (Spin(n)\times O(1))/\mathbb{Z}_2 \cong (Spin(n)\times Spin(1))/\mathbb{Z}_2\\
Spin^c(n) &= (Spin(n)\times U(1))/\mathbb{Z}_2 \cong (Spin(n)\times Spin(2))/\mathbb{Z}_2\\
Spin^h(n) &= (Spin(n)\times Sp(1))/\mathbb{Z}_2 \cong (Spin(n)\times Spin(3))/\mathbb{Z}_2.
\end{align*}
For any positive integer $k$, define $Spin^k(n) := (Spin(n)\times Spin(k))/\mathbb{Z}_2$ where the $\mathbb{Z}_2$ subgroup is generated by $(-1, -1)$. For $k = 1, 2, 3$ we see that $Spin^k(n)$ is isomorphic to $Spin(n)$, $Spin^c(n)$, and $Spin^h(n)$ respectively. For every $k$, there are short exact sequences $$0 \to \mathbb{Z}_2 \to Spin^k(n) \to SO(n)\times SO(k) \to 0$$ and $$0 \to Spin(k) \to Spin^k(n) \to SO(n) \to 0.$$

We also have pullback diagrams
\[\begin{tikzcd}
Spin^k(n) \arrow{r} \arrow{d} & Spin(n+k) \arrow{d} &  & Spin^k(n) \arrow{r} \arrow{d} & Spin(n, k)\arrow{d}\\
SO(n)\times SO(k) \arrow{r} & SO(n+k) &  & SO(n)\times SO(k) \arrow{r} & SO(n, k).
\end{tikzcd}
\] 
It follows from the second diagram that $Spin^k(n)$ is a maximal compact subgroup of $Spin(n, k)$. 

The groups $Spin^k(n)$ have been considered before \cite{EH16}, as well as in \cite{Mayer} and \cite[Section III]{McInnes} where they were denoted by $G(k, n)$ and $Spin(n)\cdot Spin(k)$ respectively. The group $Spin^4(4)$ was denoted by $\overline{\overline{Spin}}(4)$ in \cite{BFF}. More generally, for a Lie group $G$ with a central subgroup isomorphic to $\mathbb{Z}_2$, the group $(Spin(4)\times G)/\mathbb{Z}_2$ was considered in \cite{AI} and denoted $Spin_G(4)$.

\begin{definition}
A \textit{spin$^k$ structure} on a principal $SO(n)$-bundle $P$ is a principal $SO(k)$-bundle $E$ together with a principal $Spin^k(n)$-bundle $Q$ and a double covering $Q \to P\times E$ which is equivariant with respect to the homomorphism $Spin^k(n) \to SO(n)\times SO(k)$. Equivalently, a spin$^k$ structure is a reduction of structure group along the homomorphism $Spin^k(n) \to SO(n)$. We call $E$ the \textit{canonical $SO(k)$-bundle} of the spin$^k$ structure.
\end{definition}

A spin$^k$ structure on an oriented Riemannian manifold is defined to be a spin$^k$ structure on its bundle of orthonormal frames. The existence of a spin$^k$ structure is independent of the choice of orientation and Riemannian metric, so we just speak of a spin$^k$ structure on the underlying smooth manifold and call the manifold spin$^k$ if it admits such a structure.

Note that for $l < k$, a spin$^l$ structure induces a spin$^k$ structure as there is a homomorphism $Spin^l(n) \to Spin^k(n)$, induced by the natural inclusion $Spin(l) \to Spin(k)$, which commutes with the projections to $SO(n)$. Moreover, the canonical $SO(k)$-bundle of the induced spin$^k$ structure is $E\oplus\varepsilon^{k-l}$ where $E$ is the canonical $SO(l)$-bundle of the spin$^l$ structure.
 

\begin{proposition}\label{equivalences}
The following are equivalent for a smooth manifold $M$:
\begin{enumerate}
\item $M$ is spin$^k$,
\item there is an orientable rank $k$ vector bundle $E \to M$ such that $TM\oplus E$ is spin, i.e. $w_2(E) = w_2(TM)$,
\item $M$ immerses in a spin manifold with codimension $k$,
\item $M$ embeds in a spin manifold with codimension $k$.
\end{enumerate}
\end{proposition}

\begin{proof}
$(1) \Longleftrightarrow (2)$: This follows immediately from the first pullback diagram above, together with the fact that reductions of structure groups along two homomorphisms give rise to a reduction of structure group to the pullback. We also provide a second proof.

There is a commutative diagram of groups with exact rows

\[\begin{tikzcd}
0 \arrow{r} & \mathbb{Z}_2 \arrow{r} \arrow[swap]{d}{\id} & Spin(n) \arrow{r} \arrow{d} & SO(n) \arrow{d}{i_1} \arrow{r} & 0\\
0 \arrow{r} & \mathbb{Z}_2 \arrow{r} & Spin^k(n) \arrow{r} & SO(n)\times SO(k) \arrow{r} & 0\\
0 \arrow{r} & \mathbb{Z}_2 \arrow{r}\arrow{u}{\id} & Spin(k) \arrow{r} \arrow{u} & SO(k) \arrow[swap]{u}{i_2} \arrow{r} & 0
\end{tikzcd}
\]

where $i_1 \colon SO(n) \to SO(n)\times SO(k)$ and $i_2 \colon SO(k) \to SO(n)\times SO(k)$ are the natural inclusions. This gives rise to a commutative diagram in cohomology with exact rows 

\[\begin{tikzcd}
H^1(X; Spin(n)) \arrow{r} \arrow{d} & H^1(X; SO(n)) \arrow{r}{w_2} \arrow[swap]{d}{(i_1)_*} & H^2(X; \mathbb{Z}_2) \arrow{d}{\id}\\
H^1(X; Spin^k(n)) \arrow{r} & H^1(X; SO(n))\oplus H^1(X; SO(k)) \arrow{r}{\rho} & H^2(X; \mathbb{Z}_2)\\
H^1(X; Spin(k)) \arrow{r} \arrow{u} & H^1(X; SO(k)) \arrow{r}{w_2} \arrow{u}{(i_2)_*} & H^2(X; \mathbb{Z}_2) \arrow[swap]{u}{\id}.
\end{tikzcd}
\]

By commutativity, we see that $\rho([E_1] + [E_2]) = \rho([E_1]) + \rho([E_2]) = w_2(E_1) + w_2(E_2)$. It follows that $M$ admits a spin$^k$ structure if and only if there is an orientable rank $k$ bundle $E$ such that $w_2(E) = w_2(M)$.

$(2) \Longleftrightarrow (4)$: Suppose $\pi \colon E \to M$ is an orientable rank $k$ vector bundle with $w_2(E) = w_2(TM)$. It follows from the isomorphism $TE \cong \pi^*TM\oplus\pi^*E$ that the total space $E$ is a spin manifold. Note that $M$ embeds in $E$ as the zero section with codimension $k$. Conversely, if $M$ embeds in a spin manifold $N$ with codimension $k$, then one can take $E$ to be the normal bundle. Note that $E$ is an orientable rank $k$ vector bundle and $0 = w_2(TN|_M) = w_2(TM\oplus E) = w_2(TM) + w_1(TM)w_1(E) + w_2(E) = w_2(TM) + w_2(E)$, so $w_2(E) = w_2(TM)$.

$(3) \Longleftrightarrow (4)$: Clearly $(4)\Longrightarrow (3)$. For the converse, note that if $M$ immerses in a spin manifold with codimension $k$, then it embeds as the zero section of the normal bundle with codimension $k$, and the total space of the normal bundle is a spin manifold.
\end{proof}

Note, if $M$ is a closed spin$^k$ manifold, then it can be embedded in a closed spin manifold with codimension $k$. To see this, let $E$ be the normal bundle of a codimension $k$ embedding of $M$ in a spin manifold. Not only does $M$ embed in the total space of $E$, it also embeds in $S(E\oplus\varepsilon^1)$, the sphere bundle of $E\oplus\varepsilon^1$ which is a closed spin manifold with the same dimension as $E$.

\begin{example}
A closed orientable 7--manifold $M$ admits two linearly independent vector fields, see \cite[Theorem 1.1]{Th67}. If $M$ is a $G_2$ manifold, i.e. there is a reduction of structure group with respect to the inclusion $G_2 \to SO(7)$, then one can produce a third vector field via the cross product. It follows that there is an orientable rank four vector bundle $E$ which is stably equivalent to $TM$, and hence $w_2(E) = w_2(TM)$. The existence of such a bundle is not necessarily surprising as $G_2$ manifolds are spin (one could take the trivial rank four bundle). However, if $M$ is a closed torsion-free $G_2$ manifold, then $p_1(M) \neq 0$ and hence $p_1(E)\neq 0$, so $M$ has a spin$^4$ structure with non-trivial canonical $SO(4)$-bundle; in particular, such a spin$^4$ structure is not induced by a spin structure.
\end{example}

\begin{remark}
Given a homomorphism $G \to SO(n+k)$, let $H \to SO(n)\times SO(k)$ be the pullback with respect to the inclusion $SO(n)\times SO(k) \to SO(n+k)$; for $G = Spin(n+k)$, the corresponding pullback is $Spin^k(n)$. There is an induced homomorphism $H \to SO(n)$ and one can consider orientable $n$--dimensional manifolds which admit a reduction of structure group with respect to this homomorphism. For such manifolds, there is an analogue of Proposition \ref{equivalences}. 

For example, setting $n = 6$, $k = 1$, and $G = G_2$ we see that $H = G_2\cap SO(6) = SU(3)$ and we recover the fact that an orientable hypersurface of a $G_2$ manifold admits an $SU(3)$-structure; in particular, the result of Calabi \cite{Calabi} that an orientable hypersurface of $\mathbb{R}^7$ admits an almost complex structure follows immediately. Similarly, for $n = 7$, $k = 1$, and $G = Spin(7)$ we see that $H = Spin(7)\cap SO(7) = G_2$, and hence every orientable hypersurface of a $Spin(7)$ manifold is $G_2$.
\end{remark}

For any manifold $M$, there is a unique line bundle $L \to M$ with $w_1(L) = w_1(M)$ and the sphere bundle of $L$ is precisely the orientable double cover of $M$. By the above proposition, if $M$ is spin$^k$, then there is an orientable rank $k$ vector bundle $\pi \colon E \to M$ with $w_2(E) = w_2(M)$. Letting $N$ denote the sphere bundle of $E$, we have $TN\oplus\varepsilon^1 \cong TE|_N \cong \pi^*(TM\oplus E)|_N$, from which it follows that $N$ is spin. When $E$ is the canonical $SO(3)$-bundle of the spin$^h$ structure on an $(8m+4)$--dimensional almost quaternionic manifold $M$ as mentioned in Example \ref{AQ}, the resulting manifold $N$ is the twistor space\footnote{The twistor space of an $8m$--dimensional almost quaternionic manifold need not be spin in general; it will be spin if and only if the Marchiafava--Romani class is zero.} of $M$. In general, the sphere bundle $S^{k-1} \to N \to M$ can be considered a spin analogue of the orientable double cover. Just as one chooses a line bundle $L$ with $w_1(L) = w_1(M)$ when constructing the orientable double cover (as opposed to a higher rank bundle), one should take the smallest possible value of $k$ if one hopes to obtain a naturally associated spin analogue. Even then, there can be more than one orientable rank $k$ bundle $E$ with $w_2(E) = w_2(M)$, and different choices can give rise to different sphere bundles. For example, $\mathbb{CP}^2$ is spin$^c$ and $w_2(\mathcal{O}(2k+1)) = w_2(\mathbb{CP}^2)$ for every $k$, with corresponding sphere bundles $S^1 \to S^5/\mathbb{Z}_{2k+1} \to \mathbb{CP}^2$.



\begin{remark} As in the spin$^h$ case, since the homotopy fiber of $BSpin^k(n) \to BSO(n)$ is not an Eilenberg--Maclane space for $k\geq 3$, there will generally be obstructions to the existence of such a structure in multiple degrees (in contrast to the cases $k=1,2$). It is for this reason that one should also not expect a characterisation of spin${}^k$ structures in terms of low-dimensional skeleta (cf. \cite{Go97} for the case of spin and spin${}^c$).\end{remark}

We now consider how the spin$^k$ property behaves under products and connected sums.

\begin{proposition}\label{product} Let $M$ and $N$ be smooth manifolds.
\begin{enumerate} \item If $M \times N$ is spin$^k$, then $M$ and $N$ are spin$^k$. \item If $M$ is spin, then $M\times N$ is spin${}^k$ if and only if $N$ is spin${}^k$. \item If $M$ is spin${}^k$ and $N$ is spin${}^l$, then $M\times N$ is spin${}^{k+l}$.
\end{enumerate}
\end{proposition}
\begin{proof}
Throughout, let $i_M, i_N$ be natural inclusions of $M,N$ into $M\times N$, and $\pi_M, \pi_N$ the projections from $M\times N$. \newline \newline
(1) Let $E \to M \times N$ be an orientable rank $k$ vector bundle with $w_2(E) = w_2(M\times N)$. Then $i_M^*E \to M$ is an orientable rank $k$ vector bundle with $$w_2(i_M^*E) = i_M^*w_2(E) = i_M^*w_2(M\times N) = i_M^*(\pi_M^*w_2(M) + \pi_N^*w_2(N)) = w_2(M),$$ so $M$ is spin$^k$, and similarly for $N$.

(2) If $N$ is spin$^k$, then there is an orientable rank $k$ vector bundle $E \to N$ with $w_2(E) = w_2(TN)$. Now note that $$w_2(\pi_N^*E) = \pi_N^*w_2(E) = \pi_N^*w_2(TN) = w_2(\pi_N^*TN) = w_2(\pi_M^*TM \oplus \pi_N^*TN) = w_2(T(M\times N)),$$ where the penultimate equality used the fact that $M$ is spin. As $\pi_N^*E$ is orientable and has rank $k$, we see that $M\times N$ is spin$^k$. On the other hand, if $M\times N$ is spin${}^k$ then there is an orientable rank $k$ vector bundle $E \to M \times N$ with $w_2(E) = w_2(M\times N)$, so $$w_2(i_N^*E) = w_2(i_N^*T(M\times N)) = w_2(i_N^*(\pi_M^*TM \oplus \pi_N^* TN)) = w_2(\varepsilon^{\dim(M)} \oplus TN) = w_2(TN),$$ and hence $N$ is spin$^k$.

(3) Suppose $w_2(M) = w_2(E)$ and $w_2(N) = w_2(F)$ where $E \to M$ and $F \to N$ are orientable vector bundles of rank $k$ and $l$ respectively. Then $$w_2(M\times N) = \pi_M^*w_2(M) + \pi_N^*w_2(N) = \pi_M^*w_2(E) + \pi_N^*w_2(F) = w_2(\pi_M^*E\oplus\pi_N^*F)$$ and $\pi_M^*E\oplus\pi_N^*F$ has rank $k + l$.
\end{proof} 

Note that the converse of (1) is true for spin and spin$^c$ manifolds, but as Example \ref{WuxWu} demonstrates, it is not true in general: the product of spin$^h$ manifolds need not be spin$^h$.

\begin{proposition}\label{sum}
If $M$ and $N$ are spin$^k$, then $M\# N$ is also spin$^k$. 
\end{proposition}
\begin{proof}
If $M$ and $N$ are spin$^k$, there are orientable rank $k$ vector bundles $E \to M$ and $F \to N$ such that $w_2(E) = w_2(M)$ and $w_2(F) = w_2(N)$. The bundles $E$ and $F$ are classified by maps $f \colon M \to BSO(k)$ and $g \colon N \to BSO(k)$. Choosing basepoints for $M, N$, and $BSO(k)$, we can replace $f$ and $g$ by homotopic basepoint-preserving maps since $BSO(k)$ is connected. We can thus form the map $f\vee g \colon M\vee N \to BSO(k)$, and precomposing with the natural map $M\# N \to M\vee N$ gives rise to an orientable rank $k$ vector bundle with the same $w_2$ as $M\# N$.
\end{proof}

The converse is true for spin and spin$^c$; this follows from the characterisation in terms of $w_2$ and $W_3$, respectively. It is unclear whether the converse is true for $k\geq 3$.

\begin{remark} We note that the property of admitting a spin${}^k$ structure is a homotopy invariant of closed smooth manifolds. Indeed, suppose $f \colon M \to N$ is a homotopy equivalence, where $M$ and $N$ are closed and $N$ is spin$^k$ with $w_2(E) = w_2(TN)$ for $E \to N$ an orientable rank $k$ vector bundle. Then, since Stiefel--Whitney classes are homotopy invariants of closed manifolds, $M$ is orientable and we have $w_2(f^*E) = f^*w_2(E) = f^*w_2(TN) = w_2(TM)$.  \end{remark}

\subsection*{Connection between spin$^4$ and spin$^h$} We will now observe that every spin$^4$ manifold admits a spin$^h$ structure. This will allow us to extend the range of dimensions in which every orientable manifold is guaranteed to be spin$^h$.

\begin{lemma}\label{spin4=spinh}
Every spin$^4$ manifold is spin$^h$.
\end{lemma}

\begin{proof} Note that $Spin(4) \cong Spin(3)\times Spin(3)$. Projection onto the first factor induces a homomorphism $Spin^4(n) \to Spin^h(n)$ which commutes with the projections to $SO(n)$. \end{proof}

If $E$ denotes the canonical $SO(4)$-bundle of a spin$^4$ structure, then the canonical $SO(3)$-bundle of the induced spin$^h$ structure is $\Lambda^+E$ which satisfies $w_2(\Lambda^+E) = w_2(E)$, (see \cite[Proposition 2.1]{CCV}). Note that using projection onto the second factor instead, we obtain a second homomorphism $Spin^4(n) \to Spin^h(n)$. In this case, the canonical $SO(3)$-bundle is $\Lambda^-E$.



Since the $\Lambda^{\pm}$ bundles associated to an $SO(4)$-bundle $E$ satisfy $p_1(\Lambda^{\pm}E) = p_1(E) \pm 2e(E)$, where $e$ is the Euler class \cite[Lemma 2.4]{CCV}, we see that on a spin$^4$ manifold $M$, the class $w_4(M)$ has integral lifts $$\tfrac{1}{2}\left( p_1(TM) - p_1(E) \mp 2e(E) \right)$$ for any choice of $SO(4)$-bundle $E$ with $w_2(E) = w_2(M)$.

\begin{theorem}\label{upto7} The following hold: \begin{enumerate} \item Every (not necessarily compact) orientable manifold of dimension $\leq 5$ is spin$^h$. 
\item Compact orientable manifolds of dimension $6$ and $7$ are spin$^h$. \item A non-compact orientable manifold $M$ of dimension $6$ or $7$ is spin$^h$ if and only if $W_5(TM) = 0$. \item A non-compact orientable manifold $M$ of dimension $6$ or $7$ with no elements of order exactly four in $H^5(M;\ZZ)$ is spin$^h$. \end{enumerate} \end{theorem}

By ``no elements of order exactly four'' we mean that any $x \in H^5(M;\ZZ)$ satisfying $4x = 0$ also satisfies $2x = 0$.

\begin{proof} Part (1) is a combination of Proposition \ref{5manifolds} and the fact that orientable manifolds of dimension $\leq 4$ are spin$^c$ \cite{TV} and hence spin$^h$.  


That compact orientable six--manifolds immerse in $\RR^{10}$ is proved in \cite[Corollary 9]{Hir61}, and hence they admit spin$^h$ structures by Lemma \ref{spin4=spinh}. Note that the statement of \cite[Corollary 9]{Hir61} does not include compactness, though it is clear from the proof that the manifold is assumed to be closed; the general compact case then follows by taking the double if the boundary is non-empty.

For compact orientable seven--manifolds $M$, we use the result listed in the second table of \cite[p.25]{AtDu72} (see also the footnote (1) in loc. cit.), that the only obstruction to a compact seven--manifold admitting three linearly independent vector fields is the integral Bockstein of $w_4$, i.e. $W_5$. Again, the result is stated for closed manifolds, and the compact-with-boundary case follows by considering the double. (Note, for orientable seven--manifolds, $W_5$ is a priori the single obstruction to finding three linearly independent sections over the five--skeleton.) This class vanishes by \cite[Theorem 3]{Mas62}. Hence the tangent bundle of $M$ splits off a trivial rank three bundle, giving an orientable rank four bundle with the same $w_2$ as $TM$, and we again apply Lemma \ref{spin4=spinh}. This proves part (2).

Now let $M$ be a non-compact orientable six--manifold. Since $H^6(M;\ZZ) = 0$, there are no secondary or higher obstructions to admitting a spin$^h$ structure beyond $W_5$; this establishes part (3) for six--manifolds. Choose an increasing exhaustion $\{ M_i \}$ by compact manifolds with boundary.  For an abelian group $A$ and integer $k>1$, we have the short exact Milnor sequence \cite[Proposition 7.66]{Sw17}

$$0 \to {\varprojlim}^1 H^{k-1}(M_i;A) \to H^k(M;A) \to \varprojlim H^k(M_i;A) \to 0.$$ 

From the long exact sequence in cohomology associated to the short exact coefficient sequence $$0 \to \ZZ \xrightarrow{\cdot 2} \ZZ \xrightarrow{\mathrm{mod} \, 2} \ZZ/2 \to 0$$ we have the following commutative diagram:

$$ \begin{tikzcd}
0                                                  & 0                                                  & 0                                          \\
\varprojlim H^5(M_i;\ZZ) \arrow[r] \arrow[u] & \varprojlim H^5(M_i;\ZZ) \arrow[r] \arrow[u] & \varprojlim H^5(M_i;\ZZ/2) \arrow[u] \\
H^5(M;\ZZ) \arrow[u] \arrow[r, "\cdot 2"]          & H^5(M;\ZZ) \arrow[u] \arrow[r, "\mathrm{mod} \, 2"]  & H^5(M;\ZZ/2) \arrow[u]                     \\
\varprojlim^1 H^4(M;\ZZ) \arrow[u] \arrow[r]              & \varprojlim^1 H^4(M;\ZZ) \arrow[u] \arrow[r]              & \varprojlim^1 H^4(M;\ZZ/2) \arrow[u]              \\
0 \arrow[u]                                        & 0 \arrow[u]                                        & 0 \arrow[u]                               
\end{tikzcd} $$

For each $M_i$, we have $W_5(M_i) = 0$ by taking the double and applying \cite[Theorem 2]{Mas62} (or crossing the double with a circle and applying \cite[Theorem 3]{Mas62} again). Generally for an orientable manifold, the mod 2 reduction of $W_5$ is $w_5$. From here and by naturality, $w_5(M) \in H^5(M;\ZZ/2)$ maps to the zero element in $\varprojlim H^5(M_i;\ZZ/2)$. By \cite[Lemma 10.3]{MiSt74}, the term $\varprojlim^1 H^4(M;\ZZ/2)$ vanishes. Therefore $w_5(M)$ must be zero as well. Now, $W_5(M) \in H^5(M;\ZZ)$ is an element of order two which maps to $w_5(M) = 0$ by mod 2 reduction. Therefore it is in the image of the map $H^5(M;\ZZ) \xrightarrow{\cdot 2} H^5(M;\ZZ)$. Since by assumption there are no elements of order exactly four in $H^5(M;\ZZ)$, it follows that $W_5(M)$ must be the zero class. This establishes part (4) for six--manifolds.

Now let $M$ be an orientable non-compact seven--manifold. We will show that the secondary obstruction to the existence of a spin$^h$ structure vanishes, establishing parts (3) and (4). Take an exhaustion $\{M_i\}$ by compact seven--manifolds with boundary. If $W_5(M) = 0$, which will for instance be true given the torsion condition on $H^5(M;\ZZ)$ by the argument above, we can choose a lift of the classifying map of the tangent bundle $M\to BSO(7)$ to $E_1$, the second stage of the relative Postnikov tower of $BSO(4)\to BSO(7)$,

$$\begin{tikzcd}
                       & BSO(4) \arrow[d]        &                      \\
                       & \vdots \arrow[d]        &                      \\
                       & E_1 \arrow[d]           & {K(\ZZ,4)} \arrow[l] \\
M \arrow[ru] \arrow[r] & BSO(7) \arrow[r, "W_5"] & {K(\ZZ,5)}          
\end{tikzcd}$$

Restricting to the $M_i$ gives a compatible system of lifts to $E_1$. We consider now the secondary obstruction $\mathfrak{o}(M)$ to admitting three linearly independent vector fields. This is a class in $H^6(M;\pi_5(V(3,7)))$, where $V(3,7)$ is the Stiefel manifold of 3--frames in $\RR^7$. We have the following exact sequence of homotopy groups: $$ \pi_6(BSO(7)) \to \pi_5(V(3,7)) \to \pi_5(BSO(4)) \to \pi_5(BSO(7)).$$ The natural map $BSO(7)\to BSO$ is an isomorphism on $\pi_{\leq 6}$, and hence we have $\pi_5(BSO(7)) = \pi_6(BSO(7)) = 0$. Furthermore, $\pi_5(BSO(4)) \cong \pi_4(SO(4)) \cong \pi_4(S^3 \times S^3) \cong \ZZ/2 \oplus \ZZ/2$. 

Since we are fixing the lifts $M \to E_1$ and $M_i \to E_1$, and they are compatible, the secondary obstruction to lifting further to $E_2$ is natural, i.e. $\mathfrak{o}(M_i)$ is the restriction of $\mathfrak{o}(M)$. 

$$\begin{tikzcd}
                         &                                             & E_2 \arrow[d]           & {K(\ZZ/2 \oplus \ZZ/2, 5) } \arrow[l] \\
                         &                                             & E_1 \arrow[d] \arrow[r] & {K(\ZZ/2 \oplus \ZZ/2, 6) }           \\
M \arrow[rru] \arrow[rr] &                                             & BSO(7)                  &                                       \\
                         & M_i \arrow[ruu] \arrow[ru] \arrow[lu, hook] &                         &                                      
\end{tikzcd}$$

Let us now argue that $\mathfrak{o}(M_i) = 0$. We will use \cite[Theorem 1.1]{Du74}, which gives us that for \emph{any} choice of lift to $E_1$ on a \emph{closed} orientable seven--manifold, the secondary obstruction vanishes. In order to apply this to $M_i$, we consider the double $DM_i$. We will argue that the lift $M_i \xrightarrow{f_i} E_1$ (obtained by restricting $M \xrightarrow{f} E_1$) extends to a lift $DM_i \to E_1$. Then, by applying loc. cit., we will have $\mathfrak{o}(DM_i) = 0$ and hence $\mathfrak{o}(M_i) = 0$. 

$$\begin{tikzcd}
                                                        &                                     & E_1 \arrow[d] & {K(\ZZ,4)} \arrow[l] \\
M_i \arrow[rru, "f_i"] \arrow[rr] \arrow[rd, "j", hook] &                                     & BSO(7)        &                      \\
                                                        & DM_i \arrow[ruu, dashed] \arrow[ru] &               &                     
\end{tikzcd}$$

First, choose any lift $DM_i \xrightarrow{G} E_1$ of $DM_i \to BSO(7)$; this exists since $W_5$ vanishes on any closed orientable seven--manifold. Now, $f_i$ and the restriction of $G$ to $M_i$ differ by the action of an element $x$ in $[M_i, K(\ZZ,4)] = H^4(M_i;\ZZ)$ (this group acts simply transitively on the homotopy classes of lifts to $E_1$). Let us denote this by $[f_i] = x \cdot [G\vert_{M_i}]$. 

Now observe that $x$ is the restriction of a class $X \in H^4(DM_i;\ZZ)$. Namely, consider the Mayer--Vietoris sequence for the double: $$ \cdots \rightarrow H^4(DM_i;\ZZ) \rightarrow H^4(M_i;\ZZ) \oplus H^4(M_i;\ZZ) \rightarrow H^4(\partial M_i;\ZZ) \rightarrow \cdots $$ The element $(x,x)$ maps to zero, and hence $x = j^*X$ for some $X \in H^4(DM_i;\ZZ)$. 

Therefore, if we consider the (class of the) lift $X \cdot [G]$ on $DM_i$ instead of $[G]$, by naturality we have that its restriction to $M_i$ is $x \cdot [G\vert_{M_i}]$, i.e. $[f_i]$. 

Now we have that $\mathfrak{o}(M_i) = 0$ for all $i$. Consider the short exact sequence $$0 \to {\varprojlim}^1 H^{3}(M_i;\ZZ/2 \oplus \ZZ/2) \to H^4(M;\ZZ/2 \oplus \ZZ/2) \to \varprojlim H^4(M;\ZZ/2 \oplus \ZZ/2) \to 0.$$ 

Since $H^*(-;\ZZ/2 \oplus \ZZ/2)$ is naturally isomorphic to $H^*(-;\ZZ/2) \oplus H^*(-;\ZZ/2)$, the ${\varprojlim}^1$ term vanishes. Further, since $\mathfrak{o}(M)$ maps to $(\mathfrak{o}(M_i))_i$, which is the zero element, by injectivity we have that $\mathfrak{o}(M) = 0$. Since $M$ has the homotopy type of a six--complex, the secondary obstruction is also the final obstruction to admitting three linearly independent vector fields, and we conclude that $M$ admits a spin$^h$ structure. \end{proof}

\begin{remark} The primary obstruction to a spin$^c$ structure on an orientable manifold is $W_3$, and compact orientable four--manifolds are spin$^c$. An analogous argument to the above then shows that non-compact orientable four--manifolds with no elements of order exactly four in $H^3(-;\ZZ)$ are spin$^c$. The four--torsion assumption can be removed in this case \cite{TV}, and it is not clear whether one should expect this in the theorem above. \end{remark}

\section{Existence of manifolds which do not admit spin$^k$ structures}

For each $k$, we will now give an example of a closed smooth manifold which does not admit a spin${}^k$ structure. The primary tool is the following integrality theorem due to Mayer \cite[Satz 3.2]{Mayer}, which he applied to the problem of immersions in Euclidean space. A generalised integrality theorem was later obtained by B\"ar, see \cite[Theorem 9]{Bar}, who then applied it to the problem of immersions in spin manifolds with possibly additional structures on the tangent and normal bundles.

\begin{theorem}\label{integrality} \emph{(}Mayer\emph{)} Let $X$ be an even-dimensional closed orientable manifold with a codimension $k$ immersion in a spin manifold. Denoting $k=2l$ if $k$ is even or $k=2l+1$ if it is odd, we have $$\int_X 2^l \mathcal{M}(\nu) \hat{A}(X) \in \ZZ.$$ Here $\mathcal{M}(\nu)$ is the \textit{Mayer class} of the normal bundle $\nu$, defined by $\mathcal{M}(\nu) = \prod_{i=1}^l \cosh(x_j/2)$ where $p(\nu) = \prod_{j=1}^l (1+x_j^2)$ formally. \end{theorem} The integrality statement holds upon twisting by arbitrary complex vector bundles, giving us the following:

\begin{corollary}\label{fullintegrality} Let $X$ be an $2n$--dimensional closed orientable manifold equipped with a spin$^k$ structure with canonical $SO(k)$-bundle $E$. Denoting $k = 2l$ or $k=2l+1$, we have $$\int_X 2^{l} \mathcal{M}(E) \hat{A}(X) z \in \ZZ$$ for all $z$ in the integer polynomial ring generated by the elementary symmetric polynomials $e_i$ in the variables $e^{y_j} + e^{-y_j} - 2$, where $p(TX) = \prod_{j=1}^n (1+y_j^2)$. \end{corollary}

\begin{proof} By \cite[Satz 3.2]{Mayer}, the quantity $\int_X 2^{l} \mathcal{M}(\nu) \hat{A}(X) \ch(W) \in \ZZ$ is an integer for any complex vector bundle $W$ over $X$, and hence for any virtual complex vector bundle by additivity of the Chern character. Taking $W$ to be $\pi^i(TX) \otimes \CC$, where $\pi^i$ is the $i^{\textrm{th}}$ $KO$--theoretic Pontryagin class, we have that $\ch(W)$ is the $i^{\textrm{th}}$ elementary symmetric polynomial in the variables $e^{x_j}-e^{-x_j}-2$ \cite[p. 218]{MaMi79}.\end{proof}

To avoid dealing with the Mayer class, we will consider closed smooth $2n$--manifolds, whose rational Betti numbers are all trivial except for $b_0 = b_{2n} = 1$ and $b_{n}$. Let us call such manifolds \textit{rationally highly connected}. Note that a rationally highly connected $2n$--dimensional spin$^k$ manifold necessarily has trivial Mayer class if $k< \tfrac{n}{2}$, as it is a rational polynomial in the Pontryagin classes of the normal bundle, all of which (except for the zeroth class) lie in torsion cohomology groups by assumption.

In \cite[Theorem 18]{KS19}, Kennard and Su show that the signature $\sigma$ of a rationally highly connected spin manifold of dimension $8m$ satisfies the inequality $\nu_2(2\sigma) \geq 4m - 2\nu_2(m) - 5$ if it is non-zero, where $\nu_2$ is 2--adic valuation. Using Corollary \ref{fullintegrality}, we can adapt their argument to the spin${}^k$ setting, giving the following:

\begin{lemma}\label{signature} The signature $\sigma$ of a rationally highly connected $8m$--dimensional spin$^k$ manifold $X$ with $k<2m$ satisfies $$\nu_2(2\sigma) \geq 4m-5-2\nu_2(m) - l$$ if it is non-zero, where $k = 2l$ if $k$ is even or $k=2l+1$ if it is odd. \end{lemma}

\begin{proof} The argument is directly analogous to the one used in the proof of \cite[Theorem 18]{KS19}, with the modification that here $\int_X \hat{A}(TX)$ and $\int_X e_1^2 \hat{A}(TX)$ lie in $2^{-l}\ZZ$ by Corollary \ref{fullintegrality}, instead of in $\ZZ$. (We remark that the simply connected assumption therein is not used in the argument.)
\end{proof}

For a fixed $l$, note that $4m-5-2\nu_2(m) - l \to \infty$ as $m \to \infty$. So, provided $m$ is large enough, a rationally highly connected $8m$--dimensional manifold with odd signature (so that $\nu_2(2\sigma) = 1$) does not admit a spin${}^k$ structure. To build such manifolds, we use (a slightly modified version of) the condensed form of the realisation theorem for rational homotopy types found in \cite[Theorem 1]{KS19}. The existence of rationally highly connected manifolds with odd signature is known to Kreck and Zagier (see e.g. \cite{Zag17}), who proved the stronger statement that there is such a manifold in dimension $d$ if and only if $d$ is $4, 8$, or $8r$ for $r = 2^i + 2^j$. We give an outline of a construction for our particular case here for completeness. 

\begin{theorem}\label{realization}\emph{(}Kennard--Su\emph{)} Let $H$ be a rational Poincar\'e duality algebra of dimension $8m$, i.e. a graded commutative algebra over $\QQ$ with $\dim H^{8m} = 1$ such that the symmetric bilinear pairing $H^\ast \otimes H^{n-\ast} \to \QQ$ given by $\alpha \otimes \beta \mapsto \mu(\alpha\beta) =: \langle \alpha \beta, \mu \rangle$ is non-degenerate for some (and hence any) non-zero $\mu \in (H^{8m})^*$. Furthermore, suppose $H^1 = 0$. There is a closed simply connected smooth manifold $X$ whose rational cohomology algebra is isomorphic to $H$ if we can choose elements $p_m \in H^{4m}$, $p_{2m} \in H^{8m}$, and a non-zero $\mu \in (H^{8m})^*$ such that \begin{enumerate} \item the non-degenerate symmetric bilinear form on $H^{4m}$, given by $\alpha \otimes \beta \mapsto \langle \alpha \beta, \mu \rangle$, is equivalent over $\QQ$ to a diagonal form with only $\pm 1$ on the diagonal, and \item $x = \langle p_m^2, \mu \rangle$ and $y  = \langle p_{2m}, \mu \rangle$ are integers satisfying \begin{itemize} \item[($i$)] $s_{m,m} x  + s_{2m}y = \sigma$, \item[($ii$)] $\left( \dfrac{(-1)^{m+1}}{(2m-1)!} s_m + \dfrac{1}{2(4m-1)!}\right) x - \dfrac{1}{(4m-1)!} y \in \ZZ\left[\tfrac{1}{2}\right]$, \item[($iii$)] $\dfrac{1}{(2m-1)!^2} x \in \ZZ\left[\tfrac{1}{2}\right]$. \end{itemize} \end{enumerate} Here $\sigma$ is the signature of the pairing on $H^{4m}$ with respect to $\mu$, and $s_m, s_{m,m}, s_{2m} \in \QQ$ are coefficients in the Hirzebruch $L$--polynomial $1 + s_m p_m + s_{m,m}p_m^2 + s_{2m}p_{2m}$. The Pontryagin numbers $\int_X p_m^2(TX)$ and $\int_X p_{2m}(TX)$ of the obtained manifold are equal to $x$ and $y$ respectively. \end{theorem}

\begin{proof} Choosing all rational Pontryagin classes other than $p_0, p_{m}, p_{2m}$ to be zero, these conditions on $H$ are sufficient to construct a realising manifold by \cite[Theorem 13.2]{Sull77}; the second bullet point is the simplification of the Stong congruence conditions in the case where all Pontryagin classes other than $p_0, p_{m}, p_{2m}$ are torsion, as proved in \cite[Theorem 1]{KS19}. (In the context of \cite{KS19}, the theorem is stated for the realisation of rationally highly connected manifolds with $b_{4m} = 1$, and so our $x$ is an integer square in \cite[Theorem 1]{KS19} and the signature is $\pm 1$; the proof therein still gives the conclusion stated above.) \end{proof}

Now we note that for every $m$ that is a power of two, there is some odd $\sigma$ such that the above equations ($i$)-($iii$) have integer solutions $x,y$. First of all, note that we can satisfy ($ii$) and ($iii$) by choosing $x$ and $y$ to be appropriate odd integers. Replacing $x$ by a multiple, we can ensure that $s_{m,m}x$ is an integer which we can choose to be even. On the other hand, because $\nu_2(s_{2m}) = 0$ (as we will see below), we can replace $y$ by an odd multiple so that $s_{2m}y$ is an odd integer. Replacing $y$ once more by an odd multiple, and changing sign if necessary, we can ensure that ($i$) is satisfied for an arbitrarily large positive odd integer $\sigma$. 

To see that $\nu_2(s_{2m}) = 0$, we use the identity $$s_{2m} = \frac{2^{4m}(2^{4m-1}-1)}{(4m)!} B_{2m},$$ see \cite[p.12]{Hirz}, where $B_{2m}$ is the $2m^{\mathrm{th}}$ Bernoulli number, $$B_1 = \tfrac{1}{6}, B_2 = \tfrac{1}{30}, B_3 = \tfrac{1}{42}, \ldots .$$ Since $m$ is a power of 2, $\nu_2((4m)!) = 4m-1$ and $\nu_2(B_{2m}) = -1$ by the von Staudt--Clausen theorem (see e.g. \cite[Section 3]{KS19}), the claim follows by a quick calculation. 

We now address whether we can find a Poincar\'e duality algebra $H$ with fundamental class $\mu \in (H^{8m})^*$, with $m$ a power of two, satisfying condition (1), equipped with elements $p_m$ and $p_{2m}$ such that $\langle p_m^2, \mu \rangle$ and $\langle p_{2m}, \mu \rangle$ are the above integers $x$ and $y$. Suppose $\sigma > 4$, and take $H$ to be $\QQ[\alpha_1, \ldots \alpha_{\sigma}]$ modulo the relations $\alpha_i^2 = \alpha_j^2$ and $\alpha_i\alpha_j = 0$ for all $i\neq j$, where $\deg(\alpha_i) = 4m$ for all $i$. Choose $\mu \in (H^{8m})^*$ to be such that $\langle \alpha_i^2, \mu \rangle = 1$. Then condition (1) is satisfied, and the signature of $H$ with respect to $\mu$ is $\sigma$. All that remains is to choose elements $p_m, p_{2m} \in H$ such that $x = \langle p_m^2, \mu \rangle$ and $y = \langle p_{2m} , \mu \rangle$. To this end, choose $p_{2m} = y\alpha_1^2$. Then, for rational numbers $a_i$, note that $$\left\langle \left(\sum_i a_i \alpha_i \right)^2, \mu \right\rangle = \sum_i a_i^2.$$ Therefore, since $\sigma > 4$ we can find an element $p_m = \sum_i a_i \alpha_i$ such that $\langle p_m^2, \mu \rangle = x$ by Lagrange's four square theorem. 

We have thus produced a rationally highly connected manifold with odd signature in every dimension that is a power of two, and hence we have established the following:

\begin{theorem}\label{4} For every $k$ there is a closed smooth simply connected manifold not admitting a spin${}^k$ structure. \end{theorem}

By taking a connected sum of such a manifold with itself an odd number of times, we obtain infinitely many homotopy types of non-spin$^k$ manifolds in that dimension.

\begin{remark}\label{Su} In \cite{Su14}, Su constructed (infinitely many mutually non-homeomorphic) examples of rationally highly connected 32--manifolds with signature 1. It follows from Theorem \ref{signature} that these manifolds do not admit spin${}^7$ structures. \end{remark}

\subsection*{An example in dimension 8}

Note that the examples of non-spin$^h$ manifolds one obtains via the above construction have dimension at least 32. Given that every compact orientable manifold of dimension $\leq 7$ is spin$^h$ by Theorem \ref{upto7}, it is natural to ask whether there is an 8--dimensional compact non-spin$^h$ manifold.

We can indeed produce an 8--manifold not admitting a spin${}^h$ structure by directly applying the integrality statement in Theorem \ref{integrality}, whose consequence is that $$\int_X 2 \cosh\left( \tfrac{ \sqrt{p_1(E)}}{2} \right) \hat{A}(TX) \in \ZZ$$ for a manifold admitting a codimension 3 immersion in a spin manifold with normal bundle $E$ (see \cite[Theorem 5]{Bar}, also \cite[Theorem 5.1]{Nagase} for a closely related earlier statement). Indeed, the total Pontryagin class of $E$ can be written formally as $1+x^2$, so formally $x = \sqrt{p_1(E)}$, and the Mayer class is given by $\mathcal{M}(E) = \cosh \bigl( \tfrac{x}{2} \bigr) = \cosh \bigl( \tfrac{\sqrt{p_1(E)}}{2} \bigr)$; one could argue that formally $x$ could also be $-\sqrt{p_1(E)}$, but as $\cosh$ is an even function, one obtains the same result.

Take $H = \QQ[\alpha]/(\alpha^3)$ where $\deg(\alpha) = 4$, with fundamental class $\mu$ such that $\langle \alpha^2, \mu \rangle = 1$.  The signature is 1, and for $p_1 \in H^4$, $p_2 \in H^8$, one calculates the conditions on Pontryagin numbers in the rational realisation theorem (see Theorem \ref{realization}) to be that $\langle p_1^2, \mu \rangle$ and $\langle p_2, \mu \rangle$ are integers satisfying \begin{align*} 7 \langle p_2, \mu \rangle - \langle p_1^2, \mu \rangle &= 45 \\ 5 \langle p_1^2, \mu \rangle - 2 \langle p_2, \mu \rangle &\equiv 0 \bmod 3.\end{align*} Observe that the second condition follows from the first. Denoting $p_1 = x\alpha$ and $p_2 = y\alpha^2$, these conditions are thus equivalent to $x$ and $y$ being integers ($x$ must be an integer since $\langle p_1^2, \mu \rangle = x^2$ must be an integer) satisfying $$7y-x^2 = 45.$$

Suppose we have chosen such integers $x$ and $y$; then there is a closed smooth 8--manifold $X$ with $x^2$ and $y$ as its Pontryagin numbers $\int_X p_1(TX)^2$ and $\int_X p_2(TX)$. Now suppose this manifold admits a spin${}^h$ structure, with canonical $SO(3)$-bundle $E$. Recall from Proposition \ref{integralliftw4} that there is a class $\gamma \in H^4(X;\ZZ)$ such that $p_1(E) = p_1(TX) + 2\gamma$. The integrality statement above then gives us $$\int_X -\frac{p_1(TX)^2}{360} - \frac{p_2(TX)}{720} + \frac{\gamma^2}{48} = - \frac{x^2}{360} - \frac{y}{720} + \int_X \frac{\gamma^2}{48} \in \ZZ.$$ Choosing a generator of $H^4(X;\ZZ)$, whose square integrates to 1 over the manifold, we can write the free part of the class $\gamma$ as an integer multiple $c$ of this generator. Using $7y - x^2 = 45$, the integrality statement thus simplifies to $$c^2 - y + 6 \equiv 0 \bmod 48.$$ Now take for example $x = -168a + 240$ and $y = 4032a^2 - 11520a + 8235$ for any integer $a$, which satisfy $7y-x^2=45$. Since $y - 6 \equiv 21 \bmod 48$, and $21$ is not a quadratic residue modulo 48, the congruence $c^2 \equiv y - 6 \bmod 48$ does not have a solution, and thus the obtained manifold cannot in fact admit a spin${}^h$ structure. 

Varying $a$ we can obtain such manifolds with different Pontryagin numbers, and hence of different homeomorphism types. Note that $H$ could have in fact been any 8--dimensional rational Poincar\'e duality algebra with $\dim H^4 = 1$ and $H^1 = 0$ for the argument to go through; hence we obtain infinitely many homotopy types of non-spin${}^h$ 8--manifolds. 

\begin{remark}\label{obstruction} Recall that in Example \ref{WuxWu}, we showed that the product of the Wu manifold with itself is not spin$^h$ since $W_5 \neq 0$. In contrast, the above examples, being closed orientable 8--manifolds, satisfy $W_5 = 0$ \cite[p.170]{HH58}. It follows that the primary obstruction to a spin$^h$ structure is in general not the only obstruction. The vanishing of $W_5$ is equivalent to the existence of a $Spin^{\nu_4}$-structure on an orientable manifold, which is a Wu structure in the sense of \cite[Def. 3.3.1]{FSS15}. Unlike the case of $Spin^{\nu_2}$-structures, which are equivalent to spin$^c$ structures, we thus see that spin$^h$ is a more restrictive condition than $Spin^{\nu_4}$. \end{remark}

Taking products of the non-spin$^h$ 8--manifolds constructured above with spheres, and applying Proposition \ref{product} (2), we have:

\begin{theorem}\label{nonspinhall}
In every dimension $\geq 8$, there are infinitely many homotopy types of closed smooth simply connected manifolds which do not admit spin$^h$ structures.
\end{theorem}

\begin{example} Recall that any closed orientable smooth 4--manifold $M$ admits two spin${}^h$ structures, one with canonical $SO(3)$-bundle $\Lambda^+ TM$, and another with canonical $SO(3)$-bundle $\Lambda^- TM$. Since $p_1(\Lambda^{\pm}TM) = p_1(TM) \pm 2e(TM)$, where $e$ is the Euler class \cite[Lemma 2.4]{CCV}, we calculate $$\int_M 2 \cosh\left( \frac{ \sqrt{p_1(\Lambda^{\pm}TM)}}{2} \right) \hat{A}(TM) = \frac{1}{2}(\sigma(M) \pm e(M)).$$\end{example}

\begin{remark} It is worth noting that Mayer's Satz 3.2 can be interpreted as an obstruction to immersing a manifold in a spin manifold with codimension $k$ as we have done above. In that paper, Mayer not only considers the groups $G(n, k)$, which we have called $Spin^k(n)$, but also the groups $G(n, k, m) := (Spin(n)\times Spin(k)\times Spin(m))/\langle(1, -1, -1), (-1, -1, 1)\rangle$. One could consider orientable $n$--manifolds which admit a reduction of structure group with respect to the homomorphism $G(n, k, m) \to SO(n)$ given by projecting to the first factor. Admitting such a structure is equivalent to the existence of an orientable rank $k$ vector bundle $E$ such that $TM\oplus E$ is spin$^m$. In particular, Satz 3.1 can be interpreted as an obstruction to immersing a manifold in a spin$^c$ manifold with codimension $k$. When $m = 1$, we are in the case of immersing in a spin manifold; this corresponds to the fact that $G(n, k, 1) \cong G(n, k)$.

Note that if an orientable manifold $M$ admits a codimension $k$ immersion in a spin$^m$ manifold, then it follows that $M$ is spin$^{k+m}$, but the converse is not true. For example, $\mathbb{CP}^2$ is spin$^c$, i.e. spin$^{1+1}$, but if it were to admit a codimension 1 immersion in a spin manifold, then it would be spin, which it is not.\end{remark}

\section{$Pin^{k\pm}$ and other non-orientable analogues}

There are many non-orientable analogues of $Spin^k(n)$ which one can consider. For example, one can define the groups \begin{align*}Pin^{k+}(n) &:= (Pin^+(n)\times Spin(k))/\mathbb{Z}_2\\
Pin^{k-}(n) &:= (Pin^-(n)\times Spin(k))/\mathbb{Z}_2.
\end{align*}
By analogy with the spin$^k$ case, we say a (not necessarily orientable) manifold $M$ is pin$^{k\pm}$ if its tangent bundle admits a reduction of structure group with respect to $Pin^{k\pm}(n) \to O(n)$. This can be shown to be equivalent to the existence of an orientable rank $k$ vector bundle $E$ such that $TM\oplus E$ is pin$^{\pm}$, using an argument analagous to the one that appears in the proof of Proposition \ref{equivalences}. Recall, a vector bundle is pin$^+$ if $w_2 = 0$, and it is pin$^-$ if $w_2 + w_1^2 = 0$. This allows for another characterisation in terms of the existence of an orientable rank $k$ bundle $E$ with $w_2(E)$ equal to $w_2(M)$ or $w_2(M) + w_1(M)^2$ in the pin$^{k+}$ and pin$^{k-}$ cases respectively.

There are other possible groups one could consider:
\begin{align*}
Spin^{k+}(n) &:= (Spin(n)\times Pin^+(k))/\mathbb{Z}_2\\
Spin^{k-}(n) &:= (Spin(n)\times Pin^-(k))/\mathbb{Z}_2\\
Pin^{k++}(n) &:= (Pin^+(n)\times Pin^+(k))/\mathbb{Z}_2\\
Pin^{k+-}(n) &:= (Pin^+(n)\times Pin^-(k))/\mathbb{Z}_2\\
Pin^{k-+}(n) &:= (Pin^-(n)\times Pin^+(k))/\mathbb{Z}_2\\
Pin^{k--}(n) &:= (Pin^-(n)\times Pin^-(k))/\mathbb{Z}_2.
\end{align*}

Each of the corresponding manifolds can be characterised by the existence of a (not necessarily orientable) rank $k$ vector bundle with the appropriate characteristic class equal to the appropriate characteristic class of the manifold. For example, a manifold $M$ is spin$^{k+}$ if and only if it is orientable and admits a (not necessarily orientable) rank $k$ vector bundle $E$ with $w_2(E) = w_2(M)$.

Some of these structures admit a further characterisation in terms of embeddings. For instance, a manifold $M$ is pin$^{k\pm}$ if and only if it embeds in a pin$^{\pm}$ manifold with codimension $k$. On the other hand, if an $n$--dimensional manifold $M$ embeds in a spin manifold with codimension $k$, then $w_1(M) = w_1(E)$ and $w_2(M) + w_1(M)^2 = w_2(E)$ where $E$ denotes the normal bundle, so $M$ is pin$^{k-+}$; also note that $w_2(M) = w_2(E) + w_1(E)^2$, so $M$ is also pin$^{k+-}$. However, it is not necessarily the case that a pin$^{k+-}$ (or pin$^{k-+}$) manifold admits a codimension $k$ embedding into a spin manifold, since among the rank $k$ bundles $E$ satisfying the degree 2 characteristic class condition, there may not be any which also satisfy $w_1(E) = w_1(M)$.

\begin{remark} The groups $Pin^{3\pm}(n)$, i.e. $Pin^{h\pm}(n)$, have previously been considered in \cite[Proposition 9.16]{FH} and \cite[Appendix D]{SSGR} where they were denoted by $G^{\pm}_n$ and $G_{\pm}(n)$ respectively. \end{remark}

From now on we consider pin$^{k\pm}$. Similarly to Proposition \ref{product}, we see that if $M$ is a spin manifold, then $M\times N$ is pin${}^{k\pm}$ if and only if $N$ is pin${}^{k\pm}$. In addition, the connected sum of two pin${}^{k\pm}$ manifolds is again pin${}^{k\pm}$; cf. Proposition \ref{sum}. Note that a pin$^{4\pm}$ manifold is also pin$^{3\pm}$ in direct analogy with Lemma \ref{spin4=spinh}.

Suppose now that a manifold $M$ admits a codimension $k$ immersion in a spin manifold and denote the normal bundle by $E$. We consider three bundles obtained from $E$.

First consider $E\otimes\det E$. We have $w_1(E\otimes\det E) = (k + 1)w_1(E)$ and $$w_2(E\otimes\det E) = \begin{cases} w_2(E) + w_1(E)^2 & k \equiv 0, 3 \bmod 4\\ w_2(E) & k \equiv 1, 2 \bmod 4 \end{cases} = \begin{cases} w_2(M) & k \equiv 0, 3 \bmod 4\\ w_2(M) + w_1(M)^2 & k \equiv 1, 2 \bmod 4 \end{cases}.$$ 

Next consider $E\oplus\det E$, which has $w_1(E\oplus\det E) = 0$ and $w_2(E\oplus\det E) = w_2(M) + w_1(M)^2$.

Finally, we consider the bundle $(E\otimes\det E) \oplus\det E$. We have $w_1((E\otimes\det E)\oplus\det E) = kw_1(E)$ and $$w_2((E\otimes\det E)\oplus\det E) = \begin{cases} w_2(E) + w_1(E)^2 & k \equiv 2, 3 \bmod 4\\ w_2(E) & k \equiv 0, 1 \bmod 4 \end{cases} = \begin{cases} w_2(M) & k \equiv 2, 3 \bmod 4\\ w_2(M) + w_1(M)^2 & k \equiv 0, 1 \bmod 4. \end{cases}$$ 

Using these bundles, we see that in low dimensions, we have the following table:

\renewcommand{\arraystretch}{1.5}
\begin{table}[ht]
\begin{tabular}{l|cccccc}
Dimension            & 2         & 3         & 4        & 5          & 6          & 7          \\ \hline
Guaranteed structure & pin${}^-$ & pin${}^-$ & pin$^{h+}$ & \begin{tabular}{c} pin${}^{5+}$\\ pin${}^{5-}$\end{tabular} & pin${}^{5-}$ & pin${}^{6+}$
\end{tabular}
\vspace{0.5em}
\caption{The pin$^{k+}$ or pin$^{k-}$ structure guaranteed for any manifold in a given low dimension.}
\end{table}

Utilising the orientable examples of non-spin${}^k$ manifolds from the previous section, we can easily produce non-orientable manifolds admitting neither a pin$^{k+}$ nor a pin$^{k-}$ structure. Indeed, for a fixed $k$, take a non-spin${}^k$ manifold $M$, and consider the product $M\times K$ of $M$ with the Klein bottle $K$. Since $w_1(K)^2 = 0$, the product $M\times K$ admits a pin$^{k+}$ structure if and only if it admits a pin$^{k-}$ structure. Suppose it admits the former; it then follows that there is an $SO(k)$-bundle $P \to M\times K$ such that $w_2(P) = w_2(M\times K)$. Then, since $w_2(K) = 0$, we could pull back $P$ via an inclusion of $M$ into $M\times K$ as a factor, and conclude that $M$ is spin${}^k$, which is a contradiction.

\end{document}